\documentclass[11pt,twocolumn]{article}
\usepackage[centertags]{amsmath}
\usepackage{amsfonts}
\usepackage{amssymb}
\usepackage{amsthm}
\usepackage{amsgen}
\usepackage{amscd}
\usepackage{graphicx}
\usepackage{algorithmic}
\usepackage{algorithm}
\usepackage{multirow}
\usepackage{extramarks}
\usepackage{fancyhdr}
\usepackage{subfigure}
\usepackage[bookmarks=false]{hyperref}

\input{xy}
\xyoption{all}
\theoremstyle{plain}
\newtheorem{thm}{Theorem}[section]

\newtheorem{lem}[thm]{Lemma}

\newtheorem{ex}[thm]{Example}
\newtheorem{defn}[thm]{Definition}
\newtheorem{nota}[thm]{Notation}
\theoremstyle{remark}
\newtheorem{rem}{Remark}[section]
\numberwithin{equation}{section}
\numberwithin{algorithm}{section}

\begin{document}


\title{\textbf{The Number of Multistate Nested Canalyzing Functions
}\footnote{
This research was supported by the National Science Foundation under Grant Nr. CMMI-0908201.}}

\author{David Murrugarra and Reinhard Laubenbacher\\ Virginia Bioinformatics Institute and Mathematics Department\\ Virginia Polytechnic Institute and State University}
\maketitle

\abstract
Identifying features of molecular regulatory networks is an important problem in systems biology. 
It has been shown that the combinatorial logic of such networks can be captured in many cases
by special functions called nested canalyzing in the context of discrete dynamic network models. 
It was also shown that the dynamics of networks constructed from such functions has very
special properties that are consistent with what is known about molecular networks, and that 
simplify analysis. It is important to know how restrictive this class of functions is, for instance for 
the purpose of network reverse-engineering. This paper contains a formula for the number of such 
functions and a comparison to the class of all functions. In particular, it is shown that, as the number of 
variables becomes large, the ratio of the number of nested canalyzing functions to the number of
all functions converges to zero. This shows that the class of nested canalyzing functions is indeed
very restrictive, indicating that molecular networks have very special properties. The principal
tool used for this investigation is a description of these functions as polynomials and a parameterization
of the class of all such polynomials in terms of relations on their coefficients.


\section{Introduction}

A central problem of molecular systems biology is to understand the structure and dynamics of
molecular networks, such as gene regulatory, signaling, or metabolic networks. Some progress has been
made in elucidating general design principles of such networks. For instance, in  \cite{Milo} it was shown
that certain graph theoretic motifs appear far more often in the topology of regulatory network graphs than
would be expected at random. In \cite{Kauffman2003, Kauffman2004, Murrugarra} it was shown that a certain type of
Boolean regulatory logic  has the kind of dynamic properties one would expect from molecular networks.
And in \cite{Harris} it was shown that logical rules that appear in published Boolean models of
regulatory networks are overwhelmingly of this type.

These rules, so-called \emph{nested canalyzing} rules, are a special case of \emph{canalyzing}
rules, which are reminiscent of
Waddington's concept of canalyzation in gene regulation \cite{Waddington}. Nested canalyzing Boolean
rules were shown in \cite{Jarrah} to be identical with the class of unate cascade functions, which
have been studied extensively in computer engineering. They represent exactly the class of
Boolean functions that result in binary decision diagrams of shortest average path length \cite{Butler}. This in itself
has interesting implications for information processing in molecular networks. One consequence
of this result is that a formula derived earlier for the number of unate cascade functions of a given
number of variables \cite{SasaoKinoshita} applies to give a formula for the
number of nested canalyzing Boolean functions,
described in \cite{Jarrah}. A formula for the number of canalyzing Boolean functions had been given
in \cite{Win}.

Many molecular networks cannot be described using the Boolean framework, since more than one
threshold for a molecular species might be required to represent different modes of action. There are
several frameworks available for multistate discrete models, such as so-called logical models, Petri nets,
and agent-based models. It has been shown in \cite{Veliz-Cuba} and \cite{Hinkelmann} that all these
model types can be translated into the general and mathematically well-founded framework of
polynomial dynamical systems over a finite number system. In \cite{Murrugarra} the concept of nested
canalyzing logical rule has been generalized to such polynomial systems. It has been shown there,
furthermore, that a large proportion of rules in multistate discrete models are indeed nested canalyzing,
showing that this concept captures an important feature of the regulatory logic of molecular networks.

As was pointed out in \cite{Win} and \cite{Jarrah}, knowing the number of nested canalyzing rules
for a given number of input variables and for a given number of possible variable states is important
because on the one hand it provides an estimate of how plausible it is that such rules have evolved as
regulatory principles and, on the other hand, provides an estimate of how restrictive the set of rules
is.  The latter is important, for instance, for the reverse-engineering of networks.
If the set of rules is sufficiently restrictive, then the reverse-engineering problem, which is almost
always underdetermined due to limited data, becomes more tractable when restricted to 
reverse-engineering networks consisting of nested canalyzing functions. 
In this paper we present a formula for the number of nested canalyzing functions in a given number
of variables and show that the ratio of nested canalyzing functions and all multistate functions
converges to zero as the number of variables increases. We follow the approach in \cite{Jarrah} and
solve the problem within the framework of polynomial dynamical systems, which makes it
possible to frame it as a problem of counting solutions to a system of polynomial equations.


\section{Nested Canalyzing Functions}
As mentioned in the previous section, it is possible to view most discrete models within the framework
of dynamical systems over a finite number system, or finite field. For our purposes we will use the finite
fields $\mathbb{F}_p =\{0,1,\dots,p-1\}$, $p$ an arbitrary prime number,
otherwise known as $\mathbb{Z}/p$, the integers modulo $p$.
 Furthermore, we will assume that $\mathbb{F}_p$ is totally ordered under the canonical order, that is, its elements are arranged in linear increasing order, $\mathbb{F}_p = \{0 < 1<\dots<p-1\}$.
Let $\mathbb{F} = \mathbb{F}_p$ for some prime $p$.
We first recall the general definition of a nested canalyzing function in variables $x_1,\ldots , x_n$
from \cite{Murrugarra}.
The underlying idea is as follows: A rule is nested canalyzing, if there exists a variable $x$ such that, if $x$
receives certain inputs, then it by itself determines the value of the function. If $x$ does not receive
these certain inputs, then there exists another variable $y$ such that, if $y$ receives certain other inputs,
then it by itself determines the value of the function; and so on, until all variables are exhausted.

\begin{defn}\label{DefnNCF}
Let $S_i\subset \mathbb{F}, i = 1, \ldots , n,$ be subsets that satisfy the property that
each $S_i$ is a proper, nonempty subinterval of $\mathbb{F}$; that is, every element of
$\mathbb{F}$ that lies between two elements of $S_i$ in the chosen order
is also in $S_i$. Furthermore, we assume that the complement of each $S_i$ is also a
subinterval, that is, each $S_i$ can be described by a threshold $s_i$, with all elements of $S_i$
either larger or smaller than $s_i$. Let $\sigma$ be a permutation on $\{1,\dots,n\}$.
 \begin{itemize}
\item The function $f:\mathbb{F}^n\rightarrow \mathbb{F}$ is a nested canalyzing function in the variable order $x_{\sigma(1)},\dots,x_{\sigma(n)}$ with canalyzing input sets $S_1,\dots,S_n\subset \mathbb{F}$ and canalyzing output values $b_1,\dots,b_n,b_{n+1}\in\mathbb{F}$ with $b_n\neq b_{n+1}$ if it can be represented in the form
\begin{displaymath}
\begin{array}{l}
f(x_1,\dots,x_n)=\\ \\
\left\{
\begin{array}{l}
b_1\text{ if}\ x_{\sigma(1)}\in S_1\\
b_2\text{ if}\ x_{\sigma(1)}\notin S_1,x_{\sigma(2)}\in S_2\\
\vdots\\
b_n\text{ if}\ x_{\sigma(1)}\notin S_1,\dots,x_{\sigma(n)}\in S_n\\
b_{n+1}\text{ if}\ x_{\sigma(1)}\notin S_1,\dots,x_{\sigma(n)}\notin S_n
\end{array}\right.
\end{array}
\end{displaymath}
\item The function $f:\mathbb{F}^n\rightarrow \mathbb{F}$ is a nested canalyzing function if it is a nested canalyzing function in some variable order $x_{\sigma(1)},\dots,x_{\sigma(n)}$ for some permutation $\sigma$ on $\{1,\dots,n\}$.
\end{itemize}
\end{defn}

It is straightforward to verify that, if $p=2$, that is $\mathbb{F} = \{0,1\}$, then we recover the definition in \cite{Kauffman2003} of a Boolean nested canalyzing rule. As mentioned above, several important classes of multistate discrete models can be represented in the form of a dynamical system $f:\mathbb{F}^n\longrightarrow \mathbb{F}^n$, so that the concept of a nested canalyzing rule defined in this way has broad applicability.


\section{Polynomial form of nested canalyzing functions}

We now use the fact that any function $f:\mathbb{F}^n\longrightarrow \mathbb{F}$ can be expressed
as a polynomial in $n$ variables \cite[p. 369]{Lidl}. In this section we determine the polynomial form of
nested canalyzing functions. That is, we will determine relationships among the coefficients of a polynomial
that make it nested canalyzing. We follow the approach in \cite{Jarrah}.
Let $B_n$ be the set of functions from $\mathbb{F}^n$ to $\mathbb{F}$, i.e.,
$B_n=\{f:\mathbb{F}^n\longrightarrow \mathbb{F}\}$.
The set $B_n$ is endowed with an addition and multiplication that is induced from that of $\mathbb{F}$,
which makes it into a ring.
Let $I$ be the ideal of the ring of polynomials $\mathbb{F}[x_1,\dots,x_n]$ generated by the
polynomials $\{x^p_i-x_i\}$ for all $i=1,\dots,n$, where $p$ is the number of elements in $\mathbb{F}$.
There is an isomorphism between $B_n$ and the quotient ring
$\mathbb{F}[x_1,\dots,x_n]/I$ which is also isomorphic to
\begin{displaymath}
R=\biggl\{\sum_{\substack{(i_1,\dots,i_n)\\i_t\in \mathbb{F}\\t=1,\dots,n}}C_{i_1\dots i_n}x^{i_1}_1x^{i_2}_2\cdots x^{i_n}_n\biggr\}
\end{displaymath}

Now we use this identification to study nested canalyzing functions as elements of $R$.

Given a subset $S$ of $\mathbb{F}$, we will denote by $Q_{S}$ the indicator function of the complement of $S$, i.e., for $x_0\in \mathbb{F}$, let
\begin{displaymath}
Q_{S}(x_0)=\biggl\{\begin{array}{ll}
0&\text{if}\ x_0\in S\\
1&\text{if}\ x_0\notin S
\end{array}
\end{displaymath}

We will derive the polynomial form for $Q_{S}(x)$ in Lemma~\ref{Qr}. The following theorem gives the polynomial form of a nested canalyzing function.

\begin{thm}\label{ThmNCF}
Let $f$ be a function in $R$. Then
the function $f$ is nested canalyzing in the variable order $x_1,\dots,x_n$ with canalyzing input sets $S_1,\dots,S_n$ and canalyzing output values $b_1,\dots,b_n,b_{n+1}$ with $b_n\neq b_{n+1}$, if and only if it has the polynomial form

\begin{equation}\label{NCFformula}
f(x_1,\ldots,x_n)=
{\displaystyle\sum^{n-1}_{j=0}\biggl\{(b_{n-j+1}-b_{n-j}){\displaystyle\prod^{n-j}_{i=1}Q_{S_i}(x_i)}\biggr\}}+b_1
\end{equation}

where $Q_{S_i}$ is defined as in Lemma~\ref{Qr}.

\end{thm}

\begin{proof}

Let $f$ be a nested canalyzing function as in Definition~\ref{DefnNCF}, and let

\begin{displaymath}
\begin{array}{l}
g(x_1,\ldots,x_n)=
{\displaystyle\sum^{n-1}_{j=0}\biggl\{(b_{n-j+1}-b_{n-j}){\displaystyle\prod^{n-j}_{i=1}Q_{S_i}(x_i)}\biggr\}}\\ \\
\hspace{3cm}+ b_1
\end{array}
\end{displaymath}

Since $g$ has the right form to be in $R$, we can use
the isomorphism between $B_n$ and $R$, to reduce the proof to showing that
\begin{displaymath}
g(a_1,\ldots,a_n)=f(a_1,\ldots,a_n)
\end{displaymath}
for all $(a_1,\ldots,a_n)\in \mathbb{F}^n$.

If $a_1\in S_1$, then $Q_{S_1}(a_1)=0$, therefore
\begin{displaymath}
g(a_1,\ldots,a_n)=b_1\ \text{whenever}\ a_1\in S_1.
\end{displaymath}
If $a_1\notin S_1$ and $a_2\in S_2$, then $Q_{S_1}(a_1)=1$ and $Q_{S_2}(a_2)=0$, therefore
\begin{displaymath}
 g(a_1,\ldots,a_n)=(b_2-b_1)+b_1=b_2.
 \end{displaymath}

Iterating this process, if $a_1\notin S_1,a_2\notin S_2,\ldots,a_n\in S_n$, then
$Q_{S_1}(a_1)=1$, $Q_{S_2}(a_2)=1$,\dots, and $Q_{S_n}(a_n)=0$, therefore
 \begin{displaymath}
 g(a_1,\ldots,a_n)=(b_n-b_{n-1})+\dots+(b_2-b_1)+b_1=b_n.
 \end{displaymath}
 Finally, if $a_1\notin S_1,\ldots,a_n\notin S_n$, then $Q_{S_1}(a_1)=1$,
 $Q_{S_2}(a_2)=1$,\ldots, and $Q_{S_n}(a_n)=1$. Therefore,
 \begin{displaymath}
 \begin{array}{l}
 g(a_1,\dots,a_n)=\\
 (b_{n+1}-b_n)+\dots+(b_2-b_1)+b_1=b_{n+1}\\
 \end{array}
\end{displaymath}
 This completes the proof.
\end{proof}

\begin{thm}\label{ThmNCF1}
Let $f$ be a function in $R$. Then
the function $f$ is nested canalyzing in the variable order $x_{\sigma(1)},\dots,x_{\sigma(n)}$ with canalyzing input sets $S_1,\dots,S_n$ and canalyzing output values $b_1,\dots,b_n,b_{n+1}$ with $b_n\neq b_{n+1}$, if and only if it has the polynomial form

\begin{displaymath}
\begin{array}{l}
f(x_1,\cdots,x_n)=\\
{\displaystyle\sum^{n-1}_{j=0}\biggl\{(b_{n-j+1}-b_{n-j}){\displaystyle\prod^{n-j}_{i=1}Q_{s_{\sigma(i)}}(x_i)}\biggr\}}+b_1
\end{array}
\end{displaymath}

where $Q_{S_{\sigma(i)}}$ is defined as in Lemma~\ref{Qr}.

\end{thm}
\begin{proof}

The proof is very similar to the proof of Theorem~\ref{ThmNCF}.

\end{proof}


\section{The algebraic variety of nested canalyzing functions}
Here we derive a parametrization for the coefficients of any nested canalyzing function.
We will use this parametrization to derive a formula to compute the number nested canalyzing functions
for a given number of variables within a finite field in the next section.

Recall that elements of $B_n=\{f:\mathbb{F}^n\rightarrow \mathbb{F}\}$ can be seen as elements of

\begin{displaymath}
R=\biggl\{\sum_{\substack{(i_1,\dots,i_n)\\i_t\in \mathbb{F}\\t=1,\dots,n}}C_{i_1\dots i_n}x^{i_1}_1x^{i_2}_2\cdots x^{i_n}_n\biggr\}
\end{displaymath}

Now, as a vector space over $\mathbb{F}$, $R$ is isomorphic to $\mathbb{F}^{p^n}$ via the correspondence
\begin{displaymath}
\sum_{\substack{(i_1,\dots,i_n)\\i_t\in \mathbb{F}\\t=1,\dots,n}}C_{i_1\dots i_n}x^{i_1}_1x^{i_2}_2\cdots x^{i_n}_n\leftrightarrow(\dots,C_{i_1\dots i_n},\dots)
\end{displaymath}

We will identify the set of nested canalyzing functions in R with a subset $V^{ncf}$ of
$\mathbb{F}^{p^n}$ by imposing relations on the coordinates of its elements. We are going to use the following notation:

\begin{nota}

For $r\in \mathbb{F}$ and for $S\subset \mathbb{F}$,

\begin{displaymath}
\begin{array}{l}
C_{[r]}=C_{r\dots r}\\\\
C^j_{[r]\diagdown\{i\}}=C_{r\dots j\dots r}\text{ where the value $j$}\\\\
\text{ goes in the $i$th position.}\\\\
C^{i_1,\dots,i_j}_{[r]\diagdown\{1,\dots,j\}}=C_{i_1,\dots,i_ j,r\dots r}\\\\
\text{ where the values $i_1,\dots,i_j$ go in the} \\\\
$1,\dots,j$\text{ positions, respectively.}\\\\
C_{i_1\dots i_{n-j}}=C_{i_1\dots i_{n-j}0\dots0},\text{ i.e.},\ i_{n-j}\neq0\ \text{and}\\\\
i_s=0 \text{ for all } s>{n-j}.\\\\
S^c=\mathbb{F} \diagdown S, \text{ i.e., $S^c$ denote the complement of $S$}.
\end{array}
\end{displaymath}

\end{nota}

\begin{thm}\label{Parametrization}

Let $f\in B_n$ be given by
\begin{equation}\label{Eq1}
f(x_1,\dots,x_n) = \sum_{\substack{(i_1,\dots,i_n)\\i_t\in \mathbb{F}\\
t=1,\dots,n}}C_{i_1\dots i_n}x^{i_1}_1x^{i_2}_2\cdots x^{i_n}_n
\end{equation}
The polynomial $f(x_1,\dots,x_n)$ is a nested canalyzing function in the variable order $x_1,\dots,x_n$ with canalyzing input sets $S_1,\dots,S_n$ and canalyzing output values $b_1,\dots,b_{n+1}$ if and only if its coefficients
satisfy the following equations:

\begin{equation}\label{Parametrization1}
C_{i_1\dots i_{n-\mu}}=C^{p-1,\dots,p-1}_{[0]\diagdown\{1,\dots,n-\mu\}}\ {\displaystyle\prod^{n-\mu}_{j=1}C^{-1}_{[p-1]}C^{i_j}_{[p-1]\diagdown\{j\}}}\
\end{equation}

for $\mu = 0,\dots,n-1$, where

\begin{equation}\label{Parametrization2}
C_{[p-1]} = (b_{n+1}-b_{n}) (p-1)^{n} {\displaystyle\prod^{n}_{i=1} \mid S^c_i\mid }\ ,
\end{equation}

\begin{equation} \label{Parametrization3}
C^{i_j}_{[p-1]\diagdown\{j\}} = (p-1) \mid S^c_j \mid^{-1} \biggl({\displaystyle\sum_{\substack{r\in S^c_j}}r^{p-1-i_j}}\biggr) C_{[p-1]},
\end{equation}

for $i_j\neq 0, p-1$, and

\begin{equation} \label{Parametrization4}
C^{0}_{[p-1]\diagdown\{j\}} = (p-1) \mid S^c_j \mid^{-1} Q_{s_j}(0) C_{[p-1]},
\end{equation}
for  $j=1,\dots,n-1$.

\begin{equation} \label{Parametrization5}
\begin{array}{l}
C^{p-1,\dots,p-1}_{[0]\diagdown\{1,\dots,n-\mu \}}=\\\\
\biggl[{\displaystyle\prod^{n-\mu}_{i=1} (p-1) \mid S^c_i\mid}\biggr]{\displaystyle\sum^{\mu}_{j=0}\biggl\{ B_{n-j} \prod^{n-j}_{i = n-\mu+1}Q_{s_j}(0)}\biggr\},
\end{array}
\end{equation}

\begin{equation}\label{Parametrization6}
C^{p-1}_{[0]\diagdown\{1\}}C^{0}_{[p-1]\diagdown\{1\}} = C_{[p-1]}\biggl (C_{[0]} - b_1\biggr)
\end{equation}

where $Q_{S_j}(0)$ is defined as in Lemma~\ref{Qr} and $B_{n-j} = (b_{n-j+1}-b_{n-j})$ for $j=1,\dots,n-1$.

\end{thm}

The proof of Theorem \ref{Parametrization} is given in Appendix 2.
We now need to provide a similar parametrization for functions that are nested canalyzing with
respect to an arbitrary variable ordering.

\begin{thm} \label{ParametrizationWithPermutation}

Let $f\in B_n$ given by
\begin{equation}\label{Eq1}
f(x_1,\dots,x_n) = \sum_{\substack{(i_1,\dots,i_n)\\i_t\in \mathbb{F}\\t=1,\dots,n}}
C_{i_1\dots i_n}x^{i_1}_1x^{i_2}_2\cdots x^{i_n}_n.
\end{equation}
The polynomial $f(x_1,\dots,x_n)$ is a nested canalyzing function in the variable order $x_{\sigma(1)},\dots,x_{\sigma(n)}$ with canalyzing input sets $S_1,\dots,S_n$ and canalyzing output values $b_1,\dots,b_{n+1}$ if and only if,
\begin{equation}\label{Parametrization7}
\begin{array}{l}
C_{ i_1 \dots i_{n-\mu} } = \\
C^{p-1,\dots,p-1}_{[0] \diagdown \{ \sigma(i_1) \dots \sigma(i_{n-\mu}) \} } \ {\displaystyle\prod^{n-\mu}_{j=1}C^{-1}_{[p-1]}C^{ i_j  }_{[p-1]\diagdown \{ \sigma( j ) \} }}
\end{array}
\end{equation}
for $\mu = 0,\dots,n-1$, where

\begin{equation}\label{Parametrization8}
\begin{array}{l}
C_{[p-1]} = \\
(b_{n+1}-b_{n}) (p-1)^{n} {\displaystyle\prod^{n}_{i=1} \mid S^c_{ \sigma(i) } \mid },
\end{array}
\end{equation}

\begin{equation} \label{Parametrization9}
\begin{array}{l}
C^{ i_j  }_{[p-1]\diagdown \{ \sigma( j ) \}} = \\\\
(p-1) \mid S^c_{ \sigma( j ) } \mid^{-1} \biggl({\displaystyle\sum_{\substack{r\in S^c_{ \sigma( j ) } }}r^{p-1- i_j }}\biggr) C_{[p-1]},
\end{array}
\end{equation}

for $  i_j  \neq 0, p-1$, and

\begin{equation} \label{Parametrization10}
\begin{array}{l}
C^{0}_{[p-1]\diagdown \{ \sigma( j ) \} } = \\\\
(p-1) \mid S^c_{ \sigma( j ) } \mid^{-1} Q_{s_{ \sigma( j ) } }(0) C_{[p-1]},
\end{array}
\end{equation}
for $j=1,\dots,n-1$.

\begin{equation} \label{Parametrization11}
\begin{array}{l}
C^{p-1,\dots,p-1}_{[0]\diagdown \{ \sigma(i_1) \dots \sigma(i_{n-\mu}) \}} = \\
\biggl[{\displaystyle \prod^{n-\mu}_{i=1} (p-1) \mid S^c_{ \sigma( i ) } \mid}\biggr]\\
{\displaystyle\sum^{\mu}_{j=0}\biggl\{ B_{n-j} \prod^{n-j}_{i = n-\mu+1}Q_{S_ { \sigma( i ) } } (0) }\biggr\},
\end{array}
\end{equation}

\begin{equation}\label{Parametrization12}
\begin{array}{l}
C^{p-1}_{[0]\diagdown \{ \sigma( 1 ) \} }C^{0}_{[p-1]\diagdown \{ \sigma( 1 ) \}} = \\\\
C_{[p-1]}\biggl (C_{[0]} - b_1\biggr)
\end{array}
\end{equation}

where $Q_{s_{ \sigma( j ) } }(0)$ is defined as in Lemma~\ref{Qr} and

\begin{displaymath}
B_{n-j}=(b_{n-j+1}-b_{n-j}) \text{ for } j=1,\dots,n-1.
\end{displaymath}

\end{thm}

\begin{proof}
The proof follows the same line of reasoning used for the proof of Theorem~\ref{Parametrization}.
\end{proof}

\begin{rem}\label{RemComplement}

Notice from Equation~\ref{Parametrization8} that a nested canalyzing function
$f(x_1,\dots,x_n)$ with canalyzing input sets $S_1,\dots,S_n$ and canalyzing output values
$b_1,\dots,b_{n},b_{n+1}$ is also a nested canalyzing function
in the same variable order with canalyzing input sets $S_1,\dots,S^c_n$ and canalyzing output values
$b_1,\dots,b_{n+1},b_{n}$. In fact, Equation~\ref{Parametrization8} implies that

\begin{displaymath}
\begin{array}{ l}
C_{[p-1]} = (b_{n+1}-b_{n}) (p-1)^{n} {\displaystyle\prod^{n}_{i=1} \mid S^c_{ \sigma(i) } \mid } \\ \\
 = (b_{n}-b_{n+1}) \mid S_{ \sigma(n) } \mid (p-1)^{n} {\displaystyle\prod^{n-1}_{i=1} \mid S^c_{ \sigma(i) } \mid }
\end{array}
\end{displaymath}

Note that $- \mid S^c_{ \sigma(n) } \mid = \mid S_{\sigma(n)} \mid$ (mod $p$).

\end{rem}

\begin{rem} \label{RemDual}

For every nonzero $b \in \mathbb{F}$ and for every nested canalyzing function
$f(x_1,\dots,x_n)$ in the variable order $x_{\sigma(1)},\dots,x_{\sigma(n)}$
with canalyzing input sets $S_1,\dots,S_n$ and canalyzing output values $b_1,\dots,b_{n+1}$,
$f(x_1,\dots,x_n) + b$
is also a nested canalyzing function in the same variable order and
with the same canalyzing input sets $S_1,\dots,S_n$ and canalyzing output values
$b_1+b,\dots,b_{n+1} + b$. In fact, for $S_1,\dots,S_n$ and $b_1+b,\dots,b_{n+1} + b$
Equations~\ref{Parametrization7} -~\ref{Parametrization11} stay the same and
Equation~\ref{Parametrization12} becomes
\begin{displaymath}
C^{p-1}_{[0]\diagdown \{ \sigma( 1 ) \} }C^{0}_{[p-1]\diagdown \{ \sigma( 1 ) \}} = C_{[p-1]}\biggl (C_{[0]} - ( b_1 +b )\biggr)
\end{displaymath}
\end{rem}

\section{Number of nested canalyzing functions}

Here we derive a formula to compute the number of nested canalyzing functions 
in a given number of variables $n$ and a given finite field $\mathbb F$ with $p$ elements.

Let us denote the number of distinct nested canalyzing functions in $n$ variables by $NCF(n)$ and the number of distinct nested canalyzing functions that can be written as a product of $r$ nested canalyzing functions by $RNCF(n,r)$ (R for reducible). It is clear from Formula~\ref{NCFformula} that any nested canalyzing function can be written as a product of at most $n$ nested canalyzing functions. Hence $RNCF(n,r) = 0$ for all $r > n$. We will denote the number of distinct nested canalyzing functions in $n$ variables that cannot be written as a product of two or more nested canalyzing functions by $INCF(n)$
(I for irreducible) and the number of distinct nested canalyzing functions in $n$ variables that can be written
as a product of two or more nested canalyzing functions by $RNCF(n)$. Then
\begin{displaymath}
RNCF(n) = \sum^n_{r=2} RNCF(n,r)
\end{displaymath}
The following lemma relates $RNCF(n)$ and $INCF(n)$.
\begin{lem}\label{Irreducibles}
For each natural number $n$, we have
$INCF(n) = (p-1)RNCF(n)$.
\end{lem}
\begin{proof}
From Remark~\ref{RemDual}, for each $f\in RNCF(n)$, there are $p-1$ functions in $INCF(n)$. Conversely,  
Let $f\in INCF(n)$. From formula~\ref{NCFformula},
\begin{displaymath}
\begin{array}{l}
f(x_1,\ldots,x_n)=\\
{\displaystyle\sum^{n-1}_{j=0}\biggl\{(b_{n-j+1}-b_{n-j}){\displaystyle\prod^{n-j}_{i=1}Q_{S_i}(x_i)}\biggr\}}+b_1

\end{array}
\end{displaymath}

Let $b = p-b_1$. Hence 
\begin{displaymath}
\begin{array}{l}
f(x_1,\ldots,x_n)+b=\\
{\displaystyle\sum^{n-1}_{j=0}\biggl\{(b_{n-j+1}-b_{n-j}){\displaystyle\prod^{n-j}_{i=1}Q_{S_i}(x_i)}\biggr\}}\\ \\
= Q_{S_i}(x_i)\biggl[{\displaystyle\sum^{n-2}_{j=0}\biggl\{(b_{n-j+1}-b_{n-j}){\displaystyle\prod^{n-j}_{i=2}Q_{S_i}(x_i)}\biggr\}}\\
\hspace{2.5cm}+(b_2-b_1)\biggr]
\end{array}
\end{displaymath}
Therefore, any element of $INCF(n)$ can be obtained from an element of $RNCF(n)$.
\end{proof}
The following theorem gives us a formula to compute the number of
nested canalyzing functions for a given number of variables $n$.
\begin{thm} \label{NumberNCF}
The number of nested canalyzing functions in $n$ variables, denoted by $NCF(n)$, is given by
\begin{displaymath}
NCF(n) = p RNCF(n),
\end{displaymath}
where
\begin{displaymath}
RNCF(1) = (p-1)^2 ,
\end{displaymath}
\begin{displaymath}
RNCF(2) = 4 (p-1)^4 ,
\end{displaymath}
and, for $n \geqslant 3$,
\begin{displaymath}
\begin{array}{l}
RNCF(n) =  \\\\
{\displaystyle\sum^{n-1}_{r=2} \binom{n}{r-1} 2^{r-1} (p-1)^{r} RNCF(n-r+1)}\\ \\
\hspace{2cm}+ 2^{n-1}(p-1)^{n+1}(2+n(p-2)).
\end{array}
\end{displaymath}
\end{thm}
\begin{proof}
From Remark~\ref{RemDual}, in order to calculate $NCF(1)$
it is enough to calculate the number of nested canalyzing functions with
canalyzing output values $(0,b_2)$, with $b_2 \neq 0$,
and then multiply by $p$, because of the isomorphism between
the sets of vectors $\{(0,b_2)\}$ and $\{(b_1,b_2)\}$ with $b_2 \neq b_1$. This isomorphism is given by
\begin{displaymath}
(0,b_2) \xrightarrow{+(b_1,b_1)} (b_1, b_1 + b_2)
\end{displaymath}
for $b_1 = 1,\dots, p-1$.
 For output values of the form $(0,b_2)$, Formula~\ref{NCFformula} gives us
\begin{displaymath}
f(x_1) = b_2 Q_{S_1}(x_1)
\end{displaymath}
There are $p-1$ choices for $b_2$ and $p-1$ choices for $Q_{S_1}(x_1)$.
Note that we do not consider $2(p-1)$ choices for $Q_{S_1}(x_1)$ because a
nested canalyzing function with canalyzing input set $S_1$ and
canalyzing output values $(0,b_2)$ is also a nested canalyzing function
with canalyzing input set $S^c_1$ and canalyzing output values $(b_2,0)$
(see remark~\ref{RemComplement}). Therefore
\begin{displaymath}
RNCF(1) = (p-1)^2 ,
\end{displaymath}

Similarly, in order to calculate $NCF(2)$, it is enough to calculate the number of
nested canalyzing functions with canalyzing output values $(0,b_2,b_3)$, with $b_3 \neq b_2$,
and then multiply by $p$, because of the isomorphism between the
sets of vectors $\{(0,b_2,b_3)\}$ and $\{(b_1,b_2,b_3)\}$ given by
\begin{displaymath}
(0,b_2,b_3) \xrightarrow{+(b_1,b_1,b_1)} (b_1, b_1 + b_2,b_1 + b_3)
\end{displaymath}
for $b_1 = 1,\dots, p-1$.
For output values of the form $(0,0,b_3)$ formula~\ref{NCFformula} gives us
\begin{displaymath}
f(x_1) = b_3 Q_{S_1}(x_1) Q_{S_2}(x_2)
\end{displaymath}
There are $p-1$ choices for $b_3$ and $2(p-1)$ choices for each of the $Q_{S_i}(x_i)$ for $i = 1,2$.
Therefore, for $(0,0,b_3)$ there are $4(p-1)^3$ functions.

For $(0,b_2,b_3)$, where $b_2\neq0$, $b_3\neq0$, and $b_3\neq b_2$, Formula~\ref{NCFformula} give us
\begin{displaymath}
f(x_1) =  Q_{S_1}(x_1) \biggl((b_3-b_2)Q_{S_2}(x_2) + b_2 \biggr)
\end{displaymath}

There are $2(p-1)$ choices for $Q_{S_1}(x_1)$, $p-1$ choices for $b_2$,
$p-2$ choices for $b_3$, and $p-1$ choices for $Q_{S_2}(x_2)$.
Note that we do not consider $2(p-1)$ choices for $Q_{S_2}(x_2)$
because a nested canalyzing function with canalyzing input sets $S_1$, $S_2$
and canalyzing output values $(0,0,b_3)$ is also a nested canalyzing function
with canalyzing input sets $S_1$, $S^c_2$ and canalyzing output values $(0,b_3,0)$
(see Remark~\ref{RemComplement}).
Therefore, there are $2 (p-1)^3(p-2)$ functions.
If we count the number of functions after permuting the variables, we get
\begin{displaymath}
RNCF(2) = 4 (p-1)^3 + \binom{2}{1} 2(p-1)^3 (p-2),
\end{displaymath}
Simplifying the formula above we get
\begin{displaymath}
RNCF(2) = 4 (p-1)^4 .
\end{displaymath}

For $n\geqslant3$ let us compute the number of distinct nested canalyzing functions
that can be written as a product of $n$ nested canalyzing functions, $RNCF(n,n)$.
From Formula~\ref{NCFformula} it is clear that a nested canalyzing function
can be written as a product of $n$ nested canalyzing functions
if and only if the output values must have the form $(0,\dots,0,b_n,b_{n+1})$, where $b_{n+1}\neq b_n$.
First consider the case where $b_n=0$, i.e. the output values have the form
$(0,\dots,0,b_{n+1})$, with $b_{n+1}\neq 0$, for which we have:
\begin{displaymath}
f(x_1,\cdots,x_n) = b_{n+1} Q_{S_1}(x_1)\dots Q_{S_n}(x_n)
\end{displaymath}
There are $p-1$ choices for $b_{n+1}$ and $2(p-1)$ choices for each of the $Q_{S_i}(x_i)$ for $i = 1,\dots,n$.
Therefore, there are $2^n(p-1)^{n+1}$ functions. Note that if we permute the variables, we will still get the same functions.

For $(0,\dots,0,b_n,b_{n+1})$, where $b_{n+1}\neq 0$, $b_n\neq0$, and $b_{n+1}\neq b_n$,
Formula~\ref{NCFformula} gives us
\begin{displaymath}
\begin{array}{l}
f(x_1,\cdots,x_n) = \\
Q_{S_1}(x_1) \dots Q_{S_{n-1}}(x_{n-1}) \biggl\{ (b_{n+1} - b_{n})Q_{S_n}(x_n)\\
\hspace{2.5cm}+b_n \biggl\}
\end{array}
\end{displaymath}
There are $2(p-1)$ choices for each $Q_{S_i}(x_i)$ for $i = 1,\dots,n-1$,
$p-1$ choices for $b_n$, $p-2$ choices for $b_{n+1}$, and $p-1$ choices for $Q_{S_n}(x_n)$.
Note that we do not consider $2(p-1)$ choices for $Q_{S_n}(x_n)$
because a nested canalyzing function with canalyzing input sets $S_1, \dots,S_n$
and canalyzing output values $(0,\dots,0,b_n,b_{n+1})$ is also a nested canalyzing function
with canalyzing input set $S_1, \dots,S_{n-1},S^c_n$ and canalyzing output values $(0,\dots,0,b_{n+1},b_n)$
(see Remark~\ref{RemComplement}). Therefore, there are $2^{n-1} (p-1)^{n+1}(p-2)$ such functions.
If we count the number of functions after permuting the variables, we get
\begin{displaymath}
\begin{array}{l}
RNCF(n,n) = \\
2^n(p-1)^{n+1} + {\displaystyle\binom{n}{1} 2^{n-1} (p-1)^{n+1}(p-2)}\\ \\
=2^{n-1}(p-1)^{n+1}(2+n(p-2)).
\end{array}
\end{displaymath}

Let us now compute the number of distinct nested canalyzing functions that can be
written as a product of $n-r$ nested canalyzing functions, $RNCF(n,n-r)$ for $r=1,\dots,n-2$.
From Formula~\ref{NCFformula}, it is clear that a nested canalyzing function can be written
as a product of $n-1$ nested canalyzing functions if and only if the output values have the form
$(0,\dots,0,b_{n-r},\dots,b_n,b_{n+1})$, where $b_{n-r}\neq 0$ and $b_{n+1}\neq b_n$. In this
case we have:
\begin{displaymath}
\begin{array}{l}
f(x_1,\ldots,x_n) = \\
{\displaystyle\prod^{n-r-1}_{i=1}Q_{S_i}(x_i)} \\
{\displaystyle\sum^{r}_{j=0}\biggl\{(b_{n-j+1}-b_{n-j}){\displaystyle\prod^{n-j}_{i=n-r}
Q_{S_i}(x_i) + b_{n-r}}\biggr\}}
\end{array}
\end{displaymath}
There are $2(p-1)$ choices for each $Q_{S_i}(x_i)$ for $i = 1,\dots,n-r-1$
and $INCF(r+1)$ functions for the summation part.
Note that here we do have repetition coming from Remark~\ref{RemComplement}
because $b_{n-r}\neq0$. Therefore, there are $2^{n-r-1} (p-1)^{n-r-1} INCF(r+1)$ such functions.
If we count the number of functions after permuting the variables, we get
\begin{equation} \label{Product(n,n-r)}
\begin{array}{l}
RNCF(n,n-r) = \\\\
{\displaystyle\binom{n}{r+1} 2^{n-r-1}(p-1)^{n-r-1}INCF(r+1)}.
\end{array}
\end{equation}

Now since
\begin{displaymath}
\begin{array}{l}
RNCF(n) = \\
{\displaystyle\sum^n_{r=2} RNCF(n,r) = \sum^{n-2}_{r=0} RNCF(n,n-r)},
\end{array}
\end{displaymath}
we have
\begin{displaymath}
RNCF(n) = \sum^{n-2}_{r=1} RNCF(n,n-r) + RNCF(n,n).
\end{displaymath}
Now, replacing equation~\ref{Product(n,n-r)} in the previous formula, we have
\begin{displaymath}
\begin{array}{l}
RNCF(n) =  \\\\
{\displaystyle\sum^{n-2}_{r=1} \binom{n}{r+1} 2^{n-r-1} (p-1)^{n-r-1} INCF(r+1)} \\\\
\hspace{1cm}+ RNCF(n,n).
\end{array}
\end{displaymath}
If we make the change of variable $\mu = n-r$, then $\mu-1 = n-r-1$ and $r+1 = n-\mu+1$. Therefore,
\begin{displaymath}
\begin{array}{l}
RNCF(n) =  \\\\
{\displaystyle\sum^{n-1}_{\mu=2} \binom{n}{n-\mu+1} 2^{\mu-1} (p-1)^{\mu-1} INCF(n-\mu+1)}\\\\
\hspace{1cm}+ RNCF(n,n).
\end{array}
\end{displaymath}
But since,
\begin{displaymath}
\binom{n}{n-(\mu-1)} = \binom{n}{\mu-1},
\end{displaymath}
 we have
\begin{displaymath}
\begin{array}{l}
RNCF(n) =  \\\\
{\displaystyle\sum^{n-1}_{\mu=2} \binom{n}{\mu-1} 2^{\mu-1} (p-1)^{\mu-1} INCF(n-\mu+1)} \\\\
\hspace{1cm}+ RNCF(n,n),
\end{array}
\end{displaymath}
From Lemma~\ref{Irreducibles},
\begin{displaymath}
\begin{array}{l}
RNCF(n) =  \\\\
{\displaystyle\sum^{n-1}_{\mu=2} \binom{n}{\mu-1} 2^{\mu-1} (p-1)^\mu RNCF(n-\mu+1)} \\\\
\hspace{1cm}+ RNCF(n,n),
\end{array}
\end{displaymath}
where
\begin{equation}
\label{nbynreducibles}
RNCF(n,n) = 2^{n-1}(p-1)^{n+1}(2+n(p-2)).
\end{equation}
Finally, again using Lemma~\ref{Irreducibles}
\begin{displaymath}
NCF(n) = RNCF(n)+INCF(n) = pRCF(n)
\end{displaymath}
This completes the proof.
\end{proof}

\begin{ex} [Boolean case]
Jarrah et. al.~\cite{Jarrah} show that the class of Boolean nested canalyzing functions
is identical to the class of unate cascade functions. Sasao and Kinoshita~\cite{SasaoKinoshita}
found a recursive formula for the number of unate cascade functions. Therefore, the same formula
can be used to compute the number of Boolean nested canalyzing functions.
Below is the formula originally given by Sasao and Kinoshita~\cite{SasaoKinoshita} which is a particular case of our formula in Theorem~\ref{NumberNCF}, namely when $p=2$:
\begin{displaymath}
NCF(n)=2 E(n),
\end{displaymath}
where
\begin{displaymath}
E(1)=2,\ E(2)=4,
\end{displaymath}
and
\begin{displaymath}
E(n)=\sum^{n-1}_{r=2}\binom{n}{r-1}2^{r-1}E(n-r+1)+2^n.
\end{displaymath}
\end{ex}
Tables~\ref{p3} -~\ref{p5} show the number of nested canalyzing functions for $p=3,5$ and $n=1,\dots,8$.
\begin{table}
  \centering
  \caption{Number of nested canalyzing functions for $p=3$ and $n=1,\dots,8$}\label{p3}
  \begin{tabular}{| c | c |}
\hline
$n$   &NCF(n)   \\\hline
  1 & 12  \\\hline
  2 &192   \\\hline
  3 &5568   \\\hline
  4 &219648   \\\hline
  5 &10834944 \\\hline
  6 &641335296  \\\hline
  7 &44288360448   \\\hline
  8 &3495313145856   \\\hline
\end{tabular}
\end{table}

\begin{table}
  \centering
  \caption{Number of nested canalyzing functions for $p=5$ and $n=1,\dots,8$}\label{p5}
  \begin{tabular}{| c | c |}
\hline
$n$   &NCF(n)   \\\hline
  1 & 80  \\\hline
  2 &5120   \\\hline
  3 &547840   \\\hline
  4 &78561280 \\\hline
  5 &14082703360 \\\hline
  6 &3029304606720\\\hline
  7 &760232846295040 \\\hline
  8 &218043057365319680 \\\hline
\end{tabular}
\end{table}
\section{Asymptotic properties of $NCF(n)$}
In this section we examine the asymptotic properties of the formula given in Theorem \ref{NumberNCF}, i.e. we want to know the behavior of $NCF(n)$ as $n$ becomes large.

First we derive the following inequalities:
\begin{lem}\label{LemUpperBound1}
For fixed $n$ and $r=2,\dots,n-1$, we have
\begin{displaymath}
2^{r-1}(p-1)^rRNCF(n-r+1) \leq RNC(n)
\end{displaymath}
\end{lem}
\begin{proof}
It is a direct consequence of the formula given at Theorem \ref{NumberNCF}.
\end{proof}
\begin{lem}\label{LemUpperBound2}
For all natural numbers $n$, we have
\begin{displaymath}
RNCF(n,n) \leq 2^{2n} (p-1)^{n+2}.
\end{displaymath}\end{lem}
\begin{proof}
From equation~\ref{nbynreducibles},
\begin{displaymath}
\begin{array}{lcl}
RNCF(n,n) &=& 2^{n-1}(p-1)^{n+1}(2+n(p-2))\\
&\leqslant&2^{n-1}(p-1)^{n+1}(2+2^n(p-1))\\ 
&\leqslant&2^{n-1}(p-1)^{n+1}(2^{n+1}(p-1))\\
&\leqslant&2^{2n}(p-1)^{n+2}    
\end{array}
\end{displaymath}
\end{proof}
\begin{lem}\label{LemUpperBound3}
For all natural numbers $n\geqslant3$, we have
\begin{displaymath}
RNCF(n) \leq 2^{n(n-1)} (p-1)^{2n}.
\end{displaymath}
\end{lem}
\begin{proof}
We prove this by induction over $n$. First note that for $n=3$,
\begin{displaymath}
\begin{array}{l}
RNCF(3) = 
 {\displaystyle\binom{3}{1} 2^1 (p-1)^2 RNCF(2)}\\ \\
\hspace{2cm}+ 2^{2} (p-1)^{4} (2+3p-6)\\\\
= {24(p-1)^{6} + 4(p-1)^{4}(3p-4)}\\\\
\leq {2^5 (p-1)^6 + 4(p-1)^4(3(p-1)-1)}\\\\
\leq {2^5 (p-1)^6 + 2^5(p-1)^6}\\\\
= {2^6 (p-1)^6} = 2^{3(3-1)}(p-1)^{2(3)}.
\end{array}
\end{displaymath}

Now, assume
\begin{displaymath}
RNCF(n) \leq 2^{n(n-1)} (p-1)^{2n}.
\end{displaymath}
Then
\begin{displaymath}
\begin{array}{l}
RNCF(n+1) =  \\ \\
{\displaystyle\sum^{n}_{r=2} \binom{n+1}{r-1} 2^{r-1} (p-1)^{r} RNCF(n-r+2)}\\ \\
\hspace{2cm}+ RNCF(n+1,n+1)\\ \\
= {\displaystyle\binom{n+1}{1}2(p-1)^2RNCF(n)}\\ \\
+{\displaystyle\sum^{n}_{r=3} \binom{n+1}{r-1} 2^{r-1} (p-1)^{r} RNCF(n-r+2)}\\ \\
\hspace{2cm}+ RNCF(n+1,n+1)
\end{array}
\end{displaymath}
From Lemma~\ref{LemUpperBound1} we have
\begin{displaymath}
\begin{array}{l}
{\displaystyle\sum^{n}_{r=3} \binom{n+1}{r-1} 2^{r-1} (p-1)^{r} RNCF(n-r+2)}\\ \\
={\displaystyle\sum^{n-1}_{r=2} \binom{n+1}{r} 2^{r} (p-1)^{r+1} RNCF(n-r+1)}\\ \\
\leqslant{\displaystyle\sum^{n-1}_{r=2} \binom{n+1}{r} 2 (p-1) RNCF(n)}
\end{array}
\end{displaymath}
Therefore
\begin{displaymath}
\begin{array}{l}
RNCF(n+1) =  \\ \\
\leqslant {\displaystyle\binom{n+1}{1}2(p-1)^2RNCF(n)}\\ \\
+{\displaystyle\sum^{n-1}_{r=2} \binom{n+1}{r} 2 (p-1) RNCF(n)}\\ \\
\hspace{2cm}+ RNCF(n+1,n+1)\\ \\
\leqslant {\displaystyle2(p-1)^2RNCF(n)\sum^{n-1}_{r=1} \binom{n+1}{r}}\\ \\
\hspace{2cm}+ RNCF(n+1,n+1)
\end{array}
\end{displaymath}
Using the inductive hypothesis and Lemma~\ref{LemUpperBound2} we have
\begin{displaymath}
\begin{array}{l}
RNCF(n+1) =  \\ \\
\leqslant {\displaystyle2(p-1)^22^{n(n-1)} (p-1)^{2n}\sum^{n-1}_{r=1} \binom{n+1}{r}}\\ \\
\hspace{2cm}+2^{2(n+1)} (p-1)^{n+3}\\ \\
={\displaystyle2^{n(n-1)+1} (p-1)^{2(n+1)}\sum^{n-1}_{r=1} \binom{n+1}{r}}\\ \\
\hspace{2cm}+2^{2(n+1)} (p-1)^{n+3}
\end{array}
\end{displaymath}
Now since
\begin{displaymath}
\sum^{n-1}_{r=1} \binom{n+1}{r} = 2^{n+1}-n-3
\end{displaymath}
we have
\begin{displaymath}
\begin{array}{l}
RNCF(n+1) \leqslant \\ \\
2^{n(n-1)+1} (p-1)^{2(n+1)}(2^{n+1})  \\ \\
\hspace{2cm}+2^{2(n+1)} (p-1)^{n+3}\\ \\
\leqslant2^{n^2+2} (p-1)^{2(n+1)} \\ \\
\hspace{2cm}+2^{n^2+n-1} (p-1)^{2(n+1)}\\ \\
 \leqslant 2^{(n+1)n}(p-1)^{2(n+1)}.
\end{array}
\end{displaymath}
Note that $n-1\geqslant2$. This completes our proof.

\end{proof}
Let us denote the number of all possible functions on $n$ variables by $\psi(n)$. The following theorem show that the set of all nested canalyzing functions is an increasingly smaller subset of the set of all functions.
\begin{thm} \label{LimitThm}
The ratio $NCF(n)/\psi(n)$ converges to 0 as n becomes large.
\end{thm}
\begin{proof}

\begin{displaymath}
\begin{array}{lcl}
\frac{NCF(n)}{\psi(n)} &=& \frac{NCF(n)}{p^{p^n}} = \frac{pRNCF(n)}{p^{p^n}} \\ \\
&\leq &\frac{p(2^{n(n-1)}(p-1)^{2n})}{p^{p^n}}\\ \\
&\leq &\frac{2^{n(n-1)}p^{2n+1}}{p^{p^n}}\\ \\
&\leq &\frac{p^{n(n-1)}p^{2n+1}}{p^{p^n}}\\ \\
& = &\frac{p^{n^{2}+n+1}}{p^{p^{n}}}\rightarrow 0\text{ as }n\rightarrow\infty
\end{array}
\end{displaymath}
because the exponential function $p^n$ grows much faster than the quadratic function $n^{2}+n+1$ as $n$ becomes large.
\end{proof}
\section{Discussion}
The concept of a nested canalyzing rule has been shown to be
a useful approach to elucidating design principles for molecular regulatory networks,
and such rules appear very frequently in published network models. But how
much of a restriction do such rules impose on the regulatory logic of the network, that is,
how ``special" are such rules? It was shown in \cite{Murrugarra} that networks with
nested canalyzing rules have very special dynamic properties. In this paper we have
shown that nested canalyzing rules do indeed make up a very small subset of all 
possible rules. In particular, we provide an explicit formula for the number of such
rules for a given number of variables. 

This was done by translating the problem into the mathematical context of polynomial functions
over finite fields and the language of algebraic geometry. The formula we provide uses 
a parametric description of the class of all nested canalyzing polynomials as an algebraic variety.
In particular, this provides a very easy way to generate such polynomials through particular
parameter choices, which is very useful, for instance, for large-scale simulation studies. 

Another interesting aspect of the formula derived in this paper is as a future discovery tool.
The formula for Boolean nested canalyzing functions in \cite{Jarrah} was obtained essentially
through serendipity. Once the parameterization of this class was obtained it was possible to
explicitly compute the number of solutions of the parametric equations for small numbers of variables.
The resulting integer sequence, giving the number of Boolean nested canalyzing rules for small
numbers of variables was matched to the number of Boolean unate cascade functions, for which
a formula is known. It was shown in \cite{Jarrah} that the two classes of functions are in fact
identical. This is of independent interest, since Boolean unate cascade functions have been shown
to lead to binary decision diagrams with smallest average path length, suggesting that they
are very efficient in processing information. 
The formula for multistate nested canalyzing rules provides a similar opportunity. While there is no
obvious match to other function classes for small fields, it is worth, in our opinion, to pursue this
discovery approach further for larger fields. 
\section{Appendix 1}
In this appendix we derive the polynomial form for the indicator functions $Q_S$. First, for each $r\in \mathbb{F}$ we denote the indicator function of the singleton set $\{r\}$ by $P_r$, i.e., for $x_0\in \mathbb{F}$,
\begin{displaymath}
P_r(x_0)=\biggl\{\begin{array}{ll}
1&\text{if}\ x_0=r \\
0&\text{if}\ x_0\neq r
\end{array}
\end{displaymath}
The polynomial form of $P_r$ is given in the following lemma.

\begin{lem}\label{Pr}
For $r\in \mathbb{F}$, we have
\begin{displaymath}
P_r(x)=(p-1){\displaystyle\prod_{\substack{a\in \mathbb{F}\\ a\neq r}}(x-a)},
\end{displaymath}
which has the expanded form
\begin{displaymath}
\begin{array}{l}
P_r(x)=\\
(p-1)\left[ x^{p-1}+rx^{p-2}+r^2x^{p-3}+\cdots+r^{p-2}x+{\displaystyle\prod_{\substack
{a\in \mathbb{F}\\ a\neq r}}a} \right]\\
\end{array}
\end{displaymath}
\end{lem}

\begin{proof}
Let $g(x) = (p-1){\displaystyle\prod_{\substack{a\in \mathbb{F}\\ a\neq r}}(x-a)}$.
We want to prove that $g(x) = P_r(x)$ for all $x\in\mathbb{F}$.
Clearly $g(x_0)=0$ if $x_0\neq r$. It remains to prove that $g(r)=1$.
From the definition of $g$, it can be expanded as
\begin{displaymath}
\begin{array}{lcl}
g(x) &=& (p-1)(x(x-1)\cdots(x-(r-1)) \\
&&\cdots(x-(r+1))\cdots(x-(p-1))).
\end{array}
\end{displaymath}
Then
\begin{displaymath}
\begin{array}{ lcl }
 g(r)&=&(p-1)r!(-1)^{p-1-r}(p-1-r)! \\
&=&(p-1)(-1)^{r}(-1)(-2)\cdots(-r) \\
&&(-1)^{p-1-r}(p-1-r)!\\
\end{array}
\end{displaymath}

Now, since  $p - i = -i $ (mod p) for $ i = 1 , \dots , r $, we get
\begin{displaymath}
\begin{array}{ lcl }
 g(r) & = & (p-1)(-1)^{p-1}(p-1)(p-2)\cdots  \\
 &&(p-r)(p-1-r)!\\
&=&(p-1)(p-1)!  \\
&=&(p-1)(p-1) \text{  (from Wilson's Theorem) } \\
&=&1
\end{array}
\end{displaymath}
This proves the first assertion. For the second claim, from the previous formula for $P_r$, we see that  $P_r$ is a polynomial of degree $p-1$, so $P_r$ can be written in the form
\begin{displaymath}
\begin{array}{lcl}
P_r(x)&=&(p-1)[x^{p-1}+a_{p-2}x^{p-2} \\
&&+a_{p-3}x^{p-3}+\dots+a_{1}x+a_0],\\
\end{array}
\end{displaymath}
where
 \begin{displaymath}
a_{p-j} = (-1)^j {\displaystyle\sum_{\substack{b_1,\dots,b_j\in \mathbb{F}\\ b_1,\dots,b_j\neq r}}b_1\dots b_j} ,
\end{displaymath}
for  $j\in\{2,\dots,p-1\}$. Then
\begin{displaymath}
\begin{array}{ l }
a_{p-j} = \\ \\
(-1)^j{\displaystyle\sum^{p-1}_{\substack{b_1=0\\b_1\neq r}}\cdots{\displaystyle\sum^{p-1}_{\substack{b_{j-1}=0\\b_{j-1}\neq r}}}\biggl[\sum_{\substack{b_j\in \mathbb{F}\\b_j\neq r}}b_1\dots b_{j-1}b_j}\biggr]
\end{array}
\end{displaymath}
\begin{displaymath}
= (-1)^j{\displaystyle\sum^{p-1}_{\substack{b_1=0\\b_1\neq r}}\cdots{\displaystyle\sum^{p-1}_{\substack{b_{j-1}=0\\b_{j-1}\neq r}}}\biggl[b_1\dots b_{j-1}\sum_{\substack{b_j\in \mathbb{F}\\b_j\neq r}}b_j}\biggr]
\end{displaymath}
\begin{displaymath}
=(-1)^j{\displaystyle\sum^{p-1}_{\substack{b_1=0\\b_1\neq r}}\cdots{\displaystyle\sum^{p-1}_{\substack{b_{j-1}=0\\b_{j-1}\neq r}}}\biggl[b_1\dots b_{j-1}(-r)}\biggr]
\end{displaymath}
\begin{displaymath}
=(-1)^j{\displaystyle\sum^{p-1}_{\substack{b_1=0\\b_1\neq r}}\cdots{\displaystyle\sum^{p-1}_{\substack{b_{j-1}=0\\b_{j-1}\neq r}}}\biggl[b_1\dots b_{j-2}(-r)\sum_{\substack{b_{j-1}\in \mathbb{F}\\b_{j-1}\neq r}}b_{j-1}}\biggr]
\end{displaymath}
\begin{displaymath}
=(-1)^j{\displaystyle\sum^{p-1}_{\substack{b_1=0\\b_1\neq r}}\cdots{\displaystyle\sum^{p-1}_{\substack{b_{j-1}=0\\b_{j-1}\neq r}}}\biggl[b_1\dots b_{j-2}(-r)(-r)}\biggr]
\end{displaymath}
\begin{displaymath}
\begin{array}{l}
=(-1)^j(-r)^j\\\\
=r^j.
\end{array}
\end{displaymath}
Finally,
\begin{displaymath}
a_{0}=(-1)^{p-1}{\displaystyle\prod_{\substack{a\in \mathbb{F}\\ a\neq r}}a}={\displaystyle\prod_{\substack{a\in \mathbb{F}\\ a\neq r}}a}.
\end{displaymath}

\begin{rem}\label{even-exponent}
Note that $p-1$ is even for $p>2$ and $1=-1$ for $p=2$.
\end{rem}
This completes the proof.
\end{proof}

Finally, the polynomial form of $Q_S$ is given in the following lemma.

\begin{lem}\label{Qr}
For $S\subset \mathbb{F}$, we have
\begin{displaymath}
Q_{S}(x)={\displaystyle\sum_{r\in \mathbb{F}\diagdown S}P_{r}(x)}
\end{displaymath}
\end{lem}

\begin{proof}
Clearly, if $x_0\in S$, then $P_r(x_0)=0$ for all $r\in \mathbb{F}\diagdown S$. Therefore
\begin{displaymath}
Q_{S}(x_0)={\displaystyle\sum_{r\in \mathbb{F}\diagdown S}P_{r}(x_0)}=0
\end{displaymath}
Similarly, if $x_0\notin S$, then $P_{x_0}(x_0)=1$ and $P_r(x_0)=0$ for all $r\neq x_0$ in $\mathbb{F}\diagdown S$. Therefore
\begin{displaymath}
Q_{S}(x_0)={\displaystyle\sum_{r\in \mathbb{F}_p\diagdown S}P_{r}(x_0)}=1.
\end{displaymath}

This completes the proof.

\end{proof}
\section{Appendix 2}
Here we give the proof of Theorem \ref{Parametrization}.
\begin{proof}

Let us first assume that the polynomial $f$ is a nested canalyzing function
with canalyzing input sets $S_1,\dots,S_n$ and canalyzing output values $b_1,\dots,b_{n+1}$.
Then, by Theorem~\ref{ThmNCF}, $f$ can be expanded as
\begin{displaymath}
\begin{array}{l}
f(x_1,\ldots,x_n)=\\
{\displaystyle\sum^{n-1}_{j=0}\biggl\{(b_{n-j+1}-b_{n-j}){\displaystyle\prod^{n-j}_{i=1}Q_{S_i}(x_i)}\biggr\}}+b_1
\end{array}
\end{displaymath}
\begin{displaymath}
={\displaystyle\sum^{n-1}_{j=0}\biggl\{B_{n-j}\sum_{\substack{(i_1,\dots,i_{n-j})\\i_t\in \mathbb{F}\\t=1,\dots,n-j}}a^1_{i_1}\dots a^{n-j}_{i_{n-j}}x^{i_1}_1x^{i_2}_2\dots x^{i_{n-j}}_{n-j}\biggr\}}
\end{displaymath}
\begin{displaymath}
\begin{array}{l}
+b_1 + {\displaystyle\sum^{n-1}_{j=0}\biggl\{(b_{n-j+1}-b_{n-j}) a^1_{0} \dots a^{n-j}_0 \biggr\}},
\end{array}
\end{displaymath}
where $Q_{S_i}$ is defined as in Lemma~\ref{Qr}, i.e.,
\begin{displaymath}
\begin{array}{l}
Q_{S_i}(x_i)={\displaystyle\sum_{r\in S^c_i}P_{r}(x_i)}\\\\
=a^i_{p-1}x^{p-1}_i+a^i_{p-2}x^{p-2}_i+\dots+a^i_0.
\end{array}
\end{displaymath}

Now, from Lemma~\ref{Pr}, we have that
\begin{displaymath}
\begin{array}{lcl}
C_{[p-1]}&=&(b_{n+1}-b_{n})a^1_{p-1}\cdots a^n_{p-1}\\\\
&=&(b_{n+1}-b_{n}){\displaystyle\prod^{n}_{i=1} (p-1)\mid S^c_i\mid }\ ,
\end{array}
\end{displaymath}
and we have
\begin{displaymath}
\begin{array}{l}
C_{i_1\dots i_{n-\mu}}=\\
{\displaystyle\sum^{\mu}_{j=0}\biggl\{B_{n-j}a^1_{i_1}\dots a^{n-\mu}_{i_{n-\mu}} a^{n-\mu+1}_{0}\dots a^{n-j}_{0}}\biggr\}
\end{array}
\end{displaymath}
\begin{displaymath}
=a^1_{i_1}\dots a^{n-\mu}_{i_{n-\mu}}{\displaystyle\sum^{\mu}_{j=0}\biggl\{B_{n-j}a^{n-\mu+1}_{0}\dots a^{n-j}_{0}}\biggr\}
\end{displaymath}
Now, if $j\neq n$, then
\begin{displaymath}
\begin{array}{l}
C^{i_j}_{[p-1]\diagdown\{j\}}=\\
(b_{n+1}-b_n)a^1_{p-1}\dots a^{j-1}_{p-1} a^{j}_{i_j} a^{j+1}_{p-1}\dots a^{n}_{p-1}
\end{array}
\end{displaymath}
\begin{displaymath}
=(b_{n+1}-b_n)a^j_{i_j}{\displaystyle\prod^n_{\substack{i=1\\i\neq j}}a^i_{p-1}}
\end{displaymath}
So
\begin{displaymath}
\begin{array}{l}
a^j_{i_j}=\biggl[(b_{n+1}-b_n){\displaystyle\prod^n_{\substack{i=1\\i\neq j}}a^i_{p-1}}\biggr]^{-1}C^{i_j}_{[p-1]\diagdown\{j\}}
\end{array}
\end{displaymath}
for $j=1,\dots,n-1$,
and
\begin{displaymath}
\begin{array}{l}
C^{p-1,\dots,p-1}_{[0]\diagdown\{1,\dots,n-\mu \}}=\\\\
{\displaystyle\sum^{\mu}_{j=0}\biggl\{B_{n-j}a^1_{p-1}\dots a^{n-\mu}_{p-1} a^{n-\mu+1}_{0}\dots a^{n-j}_{0}}\biggr\}
\end{array}
\end{displaymath}
\begin{displaymath}
=a^1_{p-1}\dots a^{n-\mu}_{p-1}{\displaystyle\sum^{\mu}_{j=0}\biggl\{B_{n-j}a^{n-\mu+1}_{0}\dots a^{n-j}_{0}}\biggr\}
\end{displaymath}

\begin{displaymath}
=\biggl[{\displaystyle\prod^{n-\mu}_{i=1}a^i_{p-1}}\biggr]{\displaystyle\sum^{\mu}_{j=0}\biggl\{B_{n-j}a^{n-\mu+1}_{0}\dots a^{n-j}_{0}}\biggr\}
\end{displaymath}
So
\begin{displaymath}
\begin{array}{l}
\biggl[{\displaystyle\sum^{\mu}_{j=0}\biggl\{B_{n-j}a^{n-\mu+1}_{0}\dots a^{n-j}_{0}}\biggr]=\\\\
\biggl[{\displaystyle\prod^{n-\mu}_{i=1}a^i_{p-1}}\biggr]^{-1}C^{p-1,\dots,p-1}_{[0]\diagdown\{1,\dots,n-\mu \}}.
\end{array}
\end{displaymath}
Now,
\begin{displaymath}
\begin{array}{l}
C_{i_1\dots i_{n-\mu}}=\\
{\displaystyle\prod^{n-\mu}_{j=1}a^j_{i_j}}\biggl[{\displaystyle\sum^{\mu}_{j=0}\biggl\{B_{n-j}a^{n-\mu+1}_{0}\dots a^{n-j}_{0}}\biggr]
\end{array}
\end{displaymath}
\begin{displaymath}
\begin{array}{l}
={\displaystyle\prod^{n-\mu}_{j=1}\biggl\{\biggl[(b_{n+1}-b_n){\displaystyle\prod^n_{\substack{i=1\\i\neq j}}a^i_{p-1}}\biggr]^{-1}}\\
C^{i_j}_{[p-1]\diagdown\{j\}}\biggr\}\biggl[{\displaystyle\prod^{n-\mu}_{j=1}a^j_{p-1}}\biggr]^{-1}C^{p-1,\dots,p-1}_{[0]\diagdown\{1,\dots,n-\mu\}}
\end{array}
\end{displaymath}
\begin{displaymath}
\begin{array}{l}
={\displaystyle\prod^{n-\mu}_{j=1}\biggl[(b_{n+1}-b_n){\displaystyle\prod^n_{i=1}a^i_{p-1}}\biggr]^{-1}}\\
C^{p-1,\dots,p-1}_{[0]\diagdown\{1,\dots,n-\mu \}}\ {\displaystyle\prod^{n-\mu}_{j=1}C^{i_j}_{[p-1]\diagdown\{j\}}}
\end{array}
\end{displaymath}
\begin{displaymath}
={\displaystyle\prod^{n-\mu}_{j=1}\biggl[C_{[p-1]}\biggr]^{-1}}C^{p-1,\dots,p-1}_{[0]\diagdown\{1,\dots,n-\mu \}}\ {\displaystyle\prod^{n-\mu}_{j=1}C^{i_j}_{[p-1]\diagdown\{j\}}}
\end{displaymath}
\begin{displaymath}
=C^{p-1,\dots,p-1}_{[0]\diagdown\{1,\dots,n-\mu \}}\ {\displaystyle\prod^{n-\mu}_{j=1} C^{-1}_{[p-1]} C^{i_j}_{[p-1]\diagdown\{j\}}}
\end{displaymath}


Conversely, suppose Equations~\ref{Parametrization1} - \ref{Parametrization6} hold for the coefficients of the polynomial $f(x_1,\dots,x_n)$ in Equation~\ref{Eq1}.
We need to show that $f(x_1,\dots,x_n)$ is a nested canalyzing function.
Let $a^j_{p-1}=(p-1)\mid S^c_j\mid$,\\
$a^j_{i_j}=(p-1){\displaystyle\sum_{\substack{r\in S^c_j}}r^{p-1-i_j}}$,
and
$a^j_0=Q_{S_j}(0)$ for $j=1,\dots,n$.
Then
 \begin{displaymath}
C_{[p-1]} = (b_{n+1}-b_{n})\prod^n_{i=1}a^i_{p-1},
\end{displaymath}
and
\begin{equation}\label{prodwithinproof}
C^{-1}_{[p-1]}C^{i_j}_{[p-1]\diagdown\{j\}} = a^j_{i_j} [a^j_{[p-1]}]^{-1}
\end{equation}
as well as
\begin{displaymath}
\begin{array}{l}
C^{p-1,\dots,p-1}_{[0]\diagdown\{1,\dots,n-\mu \}} = \\\\
\biggr[{\displaystyle\prod^{n-\mu}_{i=1}a^i_{p-1}}\biggr]{\displaystyle\sum^{\mu}_{j=0}\biggl\{B_{n-j}a^{n-\mu+1}_{0}\dots a^{n-j}_{0}}\biggr\}
\end{array}
\end{displaymath}

Now from Equation~\ref{Parametrization1},
\begin{displaymath}
C_{i_1\dots i_{n-\mu}}=
C^{p-1,\dots,p-1}_{[0]\diagdown\{1,\dots,n-\mu \}}{\displaystyle\prod^{n-\mu}_{j=1}C^{-1}_{[p-1]}C^{i_j}_{[p-1]\diagdown\{j\}}}
\end{displaymath}
 and from Equation~\ref{prodwithinproof} we get
\begin{displaymath}
C_{i_1\dots i_{n-\mu}} = C^{p-1,\dots,p-1}_{[0]\diagdown\{1,\dots,n-\mu \}}\prod^{n-\mu}_{j=1} a^j_{i_j} [a^j_{[p-1]}]^{-1}.
\end{displaymath}
Then
\begin{displaymath}
\begin{array}{l}
C_{i_1\dots i_{n-\mu}} = \\\\
\biggr[{\displaystyle\prod^{n-\mu}_{i=1}a^i_{i_j}}\biggr]{\displaystyle\sum^{\mu}_{j=0}\biggl\{B_{n-j}a^{n-\mu+1}_{0}\dots a^{n-j}_{0}}\biggr\}
\end{array}
\end{displaymath}
Now, from Equation~\ref{Parametrization6} we get
\begin{displaymath}
C_{[0]} - b_1 = C^{p-1}_{[0]\diagdown\{1\}}C^{0}_{[p-1]\diagdown\{1\}}  C^{-1}_{[p-1]}.
\end{displaymath}
Then
\begin{displaymath}
C_{[0]} = b_1 + {\displaystyle\sum^{n-1}_{j=0}\biggl\{B_{n-j} a^1_{0} \dots a^{n-j}_0 \biggr\}}
\end{displaymath}
\begin{displaymath}
\begin{array}{l}
f(x_1,\cdots,x_n) = \\\\
{\displaystyle\sum^{n-1}_{\mu=0}\biggl\{ {\displaystyle\sum_{\substack{(i_1,\dots,i_{n-\mu})\\i_t\in \mathbb{F}_p\\t=1,\dots,n-\mu}}C_{i_1\dots i_{n-\mu}} x^{i_1}_1\dots x^{i_{n-j}}_{n-j}}\biggr\}}\\\\
\hspace{1cm}+C_{[0]}.
\end{array}
\end{displaymath}
Therefore
\begin{displaymath}
\begin{array}{l}
f(x_1,\cdots,x_n) = \\ \\
{\displaystyle\sum^{n-1}_{j=0}\biggl\{B_{n-j}\sum_{\substack{(i_1,\dots,i_{n-j})\\i_t\in \mathbb{F}\\t=1,\dots,n-j}}a^1_{i_1}\dots a^{n-j}_{i_{n-j}}x^{i_1}_1x^{i_2}_2\dots x^{i_{n-j}}_{n-j}\biggr\}}
\end{array}
\end{displaymath}
\begin{displaymath}
+b_1 + {\displaystyle\sum^{n-1}_{j=0}\biggl\{(b_{n-j+1}-b_{n-j}) a^1_{0} \dots a^{n-j}_0 \biggr\}}.
\end{displaymath}
Finally,
\begin{displaymath}
f(x_1,\dots,x_n)=
{\displaystyle\sum^{n-1}_{j=0}\biggl\{B_{n-j}{\displaystyle\prod^{n-j}_{i=1}Q_{S_i}(x_i)}\biggr\}}+b_1,
\end{displaymath}
which is nested canalyzing by Theorem~\ref{ThmNCF}. This completes the proof.
\end{proof}



\begin{thebibliography}{6}

 \bibitem{Butler}
 J.T. Butler et al. (2005) Average path length of binary decision diagrams.
 {\it IEEE Trans. Comput.} {\bf 54}:1041--1053.

 \bibitem{Harris}
    S.~Harris et al. (2002)
    A model of transcriptional regulatory networks based on biases in the observed regulatory rules.
    {\it Complexity} 7 (4), 23-40.

\bibitem{Hinkelmann}
    Hinkelmann, F., Murrugarra, D., Jarrah, A.S., and Laubenbacher, R. (2010)
     A mathematical framework for agent-based models of complex
     biological networks. {\it Bull. Math. Biol.}, in press.

 \bibitem{Jarrah}
Abdul~Salam Jarrah, Blessilda Raposa, and Reinhard Laubenbacher.
\newblock Nested canalyzing, unate cascade, and polynomial functions.
\newblock {\em Physica D: Nonlinear Phenomena}, 233(2):167 -- 174, 2007.

\bibitem{Win} Winfried Just, Ilya Shmulevich, John Konvalina,
\newblock ``The number and probability of canalyzing functions'',
\newblock {\em Physica D} 197 (2004), pp. 211-221.

 \bibitem{Kauffman2003}
Stuart Kauffman, Carsten Peterson, B.~Samuelsson, and Carl Troein.
\newblock Random boolean network models and the yeast transcriptional network.
\newblock {\em Proceedings of the National Academy of Sciences of the United
  States of America}, 100(25):14796--14799, 2003.

\bibitem{Kauffman2004}
Stuart Kauffman, Carsten Peterson, Björn Samuelsson, and Carl Troein.
\newblock Genetic networks with canalyzing boolean rules are always stable.
\newblock {\em Proceedings of the National Academy of Sciences of the United
  States of America}, 101(49):17102--17107, 2004.

\bibitem{Lidl}
R. Lidl and H. Niederreiter, \emph{Finite Fields}, Cambridge University Press, New York, 1997.

\bibitem{Milo}
  R. Milo et al. (2004) Superfamilies of evolved and designed networks, {\it Science} {\bf 303}:1538--1542.

  \bibitem{Murrugarra}
  D. Murrugarra and R. Laubenbacher. (2011)
  Regulatory patterns in molecular interaction networks. Under review.

\bibitem{Nikolajewa}
    S.~Nikolajewa, et al.
    \newblock {\em Boolean networks with biologically relevant rules show ordered behavior}.
    \newblock BioSystems 90 (2007) 40-47.

\bibitem{SasaoKinoshita}
T.~Sasao, et. al.
\newblock{\em On the Number of Fanout-Free Functions and Unate Cascade Functions}.
\newblock IEEE trans. Comput. 28 (1) (1979) 66-72.


\bibitem{Veliz-Cuba}
Alan Veliz-Cuba, Abdul~Salam Jarrah, and Reinhard Laubenbacher.
\newblock {Polynomial algebra of discrete models in systems biology}.
\newblock {\em Bioinformatics}, 26(13):1637--1643, 2010.

    \bibitem{Waddington}
C.H. Waddington.
\newblock Canalization of development and the inheritance of acquired
  characters.
\newblock {\em Nature}, 150:563--565, November 1942.

\end{thebibliography}
\end{document}